\documentclass[a4paper,11pt,twoside]{paper}

\usepackage{amssymb}

\usepackage{mathrsfs}
\usepackage{amsfonts,amsmath,lmodern,fancyhdr,lastpage,graphicx,nccfoots,caption}
 
\usepackage{amssymb,amsmath, amsthm}
\usepackage{graphicx}
\usepackage{graphics}
\usepackage{psfrag}
\usepackage[applemac]{inputenc}

\newtheorem{cor}{Corollary}[section]
\newtheorem{proposition}{Proposition}[section] 
\newtheorem{definition}{Definition}[section]

\newtheorem{remark}{Remark}[section]

\newcommand{\Cb}{\mathbb{C}}
\newcommand{\Rb}{\mathbb{R}}

\newcommand{\Zb}{\mathbb{Z}}

\theoremstyle{definition}

\newtheorem{exemplo}{Example}[section]

\usepackage{hyperref}

\vfuzz \hfuzz
\begin{document}
\thispagestyle{empty}


\begin{center}
\Large
\textsc{Construction of helices from  Lucas and Fibonacci sequences}
\end{center}


\begin{center}
\textit{Mário M. Gra\c{c}a}, \par
\vspace{3mm}
\begin{tabular}{l}
Departamento de Matem\'{a}tica,\\
Instituto Superior T\'ecnico, Universidade  de Lisboa,\\
 Av. Rovisco Pais,                 
1049--001 Lisboa, Portugal.    \\
e-mail: \texttt{\url{mgraca@math.tecnico.ulisboa.pt}}   
\end{tabular}
\end{center}
\today{}

\medskip
\noindent
\textbf{Abstract:} By means of two complex\--valued functions (depending on an integer parameter $P\geq 1$) we construct  helices of integer ratio $R\geq 1$  related to the so\--called Binet formulae for P\--Lucas and P\--Fibonacci  sequences. Based on these functions a new map is defined and we show that its three\--dimensional representation is also a helix.  After proving that the lattice points of these later helix satisfy certain diophantine Pell's equations we call it a Pell's helix.
 We prove that for P\--Fibonacci  and Pell's helices the respective ratio is an invariant,  contrasting to the P\--Lucas helices whose ratio depends on $P$. It is also shown that suitable linear combinations of certain  complex\--valued maps lead to new helices related to Lucas/Fibonacci/Pell numbers. Graphical examples are given in order to illustrate the underlying theory.


 \section{Introduction }\label{introd}
 In Section \ref{sec1} we  define two complex\--valued functions $g$ and $h$ which are extended versions of the so\--called sequence Binet formulae  \cite{hoggatt,devlin},  for  two well\--known families  of sequences called  P\--Lucas and P\--Fibonacci sequences,
  where $P\geq 1$ is an integer parameter (cf. \eqref{eq1} and \eqref{eq4}).  
  
  \noindent
  The maps $g$ and $h$ give rise to new complex\--valued functions $\psi_1$, $\psi_2$ (see \eqref{eq13})  for which we  show that their parametric representation in the three\--dimensional space is a helix. We call these helices  P\--Lucas and P\--Fibonacci helices, respectively.
 
 \medskip
 \noindent
 In Proposition \ref{prop1} we prove that all P\--Fibonacci helices have ratio $R=1$, while the P\--Lucas helices have ratio $R=P^2+4\geq 5$, all these helices  having the same pitch $p=2$. The ratio invariance property of P\--Fibonacci helices is fundamental to distinguish these helices from the P\--Lucas ones.

 \medskip
\noindent
 The fact of   the P\--Fibonacci's helices are ratio invariant, whose ratio $R=1$ is minimal relatively to the ratio of the P\--Lucas helices, may  bring a new point of view to a classic controversy on the so\--called omnipresence of the \lq\lq{divine proportion}\rq\rq\ and other \lq\lq{metallic means}\rq\rq
\ (see  for instance \cite{huylebrouck} and bibliography therein) as well as to other issues relevant in the phyllotaxis area of research \cite{adam, stewart}. 
 
\medskip 
 \noindent
In Section \ref{sec2} 
we define another complex\--valued function $\psi_3$ depending on $g$ and $h$. This map enables us to make a connection to certain Pell's equations (see for instance \cite{lenstra}, \cite{cohen} Ch. 6.8). 

\medskip
\noindent
The main results are propositions \ref{prop3} and \ref{prop4}. In Propo\-si\-tion~\ref{prop3}  we  prove that the curves  represented by $\psi_3$  are also ratio invariant helices ($R=4$). We name them  P\--Pell helices.

\medskip
\noindent
 Of course, any expression involving suitable arrangements of the integer terms of Lucas/Fibonacci  sequences (in the literature one can find a lot of such expressions) can be rewritten as a complex\--valued map, say $\psi$, whose geometrical properties might be unveiled using a suitable parametrization of $\psi$.  However, we restrict the scope only to three complex\--valued functions $\psi_1$  $ \psi_2$, defined in \eqref{eq13} (Lucas/Fibonacci) and $\psi_3$, defined in \eqref{eq21} (Pell) for which one can easily prove that they are representable as helices. 
 
 \medskip
 \noindent
 In Proposition \ref{prop4} we define other maps $\psi_4$ to $\psi_7$,  resulting from suitable linear combinations of $\psi_1$ to $\psi_3$,  whose representative curve is also a helix. We conclude that by appropriate linear combinations of the referred complex\--valued maps we are lead to innumerable  other helices related to Lucas/Fibonacci/Pell numbers.

 \medskip
 \noindent
The examples given in Section \ref{sec3} aim to illustrate that the representation of Lucas/Fibonnaci sequences by the complex maps $g$, $h$ and the subsequent maps $\psi_1$, $\psi_2$ and $\psi_3$, all provide tools to highlight  geometrical aspects hidden either in their recursive definition or in the referred classical Binet formulae, expanding the findings  for instance in \cite{falcon}, \cite{stakhov1} and \cite{edson}. 

\medskip
\noindent
In  examples \ref{exemplo1}\--\ref{exemplo4}  some Lucas/Fibonacci/Pell helices are displayed, from which a panoply of double helices can be constructed, as in Example \ref{exemplo3}, suggesting that Lucas/Fibonacci/Pell helices can perhaps  find several applications within the scientific \cite{watson, stewart}, artistic  \cite{heritage} or industrial  \cite{bridge} frameworks.

\medskip
\noindent
We hope that this preliminary work, based on elementary mathematics,  may be a source of inspiration in future studies. The approach,  following the spirit of \cite{needham}, seems to be a unifying setting to  deal with geometrical and analytical aspects namely for difference equations.  
  

\section{Lucas\--Fibonacci sequences}\label{sec1}
 The family of P\--Lucas  sequences  is  defined by
\begin{equation}\label{eq1}
\begin{array}{l}
L_0=2\\
L_1=P\\
L_{n+2}= P\, L_{n+1}+ L_n, \qquad n=0,1,2,\ldots,
\end{array}
\end{equation}
where $P\geq 1$ is an integer.    The characteristic polynomial associated to the difference equation \eqref{eq1} is
 \begin{equation}\label{eq2}
 p(\lambda)=\lambda^2-P\, \lambda-1,
 \end{equation}
whose real roots  are
\begin{equation}\label{eq3}
\lambda_1=\displaystyle \frac{P+\sqrt D}{2}>0, \qquad \lambda_2=\displaystyle \frac{P-\sqrt D}{2}<0, \quad \mbox{with}\quad  D=P^2+4\ .
\end{equation}
The family of P\--Fibonacci sequences is defined by same difference equation of the P\--Lucas family with different initial values. Namely,
\begin{equation}\label{eq5}
\begin{array}{l}
F_0=0\\
F_1=1\\
F_{n+2}= P\, F_{n+1}+ F_n, \qquad n=0,1,2,\ldots \ .
\end{array}
\end{equation}
Both  families share the same characteristic polynomial \eqref{eq2}, and its roots satisfy the following properties which will play a crucial role later:
\begin{equation}\label{eq4}
(i)\quad \lambda_1+\lambda_2=P,\quad 
(ii)\quad \lambda_1-\lambda_2=\sqrt{D}, \quad
(iii)\quad \lambda_1\, \lambda_2= -1\ .
\end{equation}

\medskip
\noindent
 Let us define two complex\--valued functions modelling respectively the diference equations \eqref{eq1} and \eqref{eq5}.
 \begin{definition}\label{def1}
 Given an integer $P\geq 1$, we define the  functions
 
 \noindent
   $g, h: \Rb\mapsto \Cb,$ by
  \begin{equation}\label{eq6} 
 \begin{array}{l}
 g(t)=\lambda_1^t+ \lambda_2^t,  \qquad t\in \Rb,\\
 \\
 h(t)=\displaystyle \frac{1}{\sqrt{D}}\left( \lambda_1^t-\lambda_2^t \right), \quad t\in \Rb\ .
 \end{array}
 \end{equation}
 
\noindent
The functions $g$ and $h$  are   called the (complex) representation of the P\--Lucas sequence and P\--Fibonacci sequence, respectively.
 \end{definition}

\noindent
Note that by \eqref{eq4}\--(iii) we have $\lambda_2=-1/\lambda_1<0$, and so
\begin{equation}\label{eq7}
\lambda_2^t=\displaystyle\frac{1}{\lambda_1^t}\, \left(  -1\right)^t\ .
\end{equation}
Applying the Euler's formula $e^{i \pi}=-1$, where $i=\sqrt{-1}$,  we obtain
\begin{equation}\label{eq8}
\lambda_2^t=\displaystyle\frac{e^{i\, \pi t}}{\lambda_1^t}= \displaystyle \frac{\cos(\pi\, t)+ i\, \sin(\pi\, t)}{\lambda_1^t}\ .
\end{equation}
By  \eqref{eq8} one concludes  that both $g$ and $h$, given by \eqref{eq6},  can be seen as continuous complex\--valued functions defined in $\Rb$. 

\medskip
\noindent
In the sequel the (complex) representation of  P\--Lucas or Fibonacci sequences will be used to illustrate geometrical aspects hidden in their integer recursive definitions \eqref{eq1} or \eqref{eq5}. When $P$ is assumed to be a given constant,  P\--Lucas or P\--Fibonacci sequences will be simply called  Lucas or Fibonacci sequences, respectively.

\medskip
 \noindent
Restricting the real variable $t$ in \eqref{eq6} to the non\--negative integers $\Zb_0^{+}$, the values of  $g$ and $h$ coincide with the
 well\--known and widely used  Binet\footnote{Abraham de Moivre (1667\--1754) first discovered the now called Binet formula also authored by Jacques\--Philippe Marie Binet (1786\--1856). }   formulae, respectively
 \begin{equation}\label{eq9}
 \begin{array}{l}
 L_n=\displaystyle \frac{1}{2^n}\left(\displaystyle (P+\sqrt{D})^n+ (P-\sqrt{D})^n\right),\quad n=0,1,\ldots\\
 \\
F_n=\displaystyle \frac{1}{\sqrt{D}\, 2^n}\left(\displaystyle (P+\sqrt{D})^n- (P-\sqrt{D})^n\right),\quad n=0,1,\ldots \ .\\
\end{array}
 \end{equation}
Given an integer $n$  and $D=P^2+4$ not a perfect square, the expressions in \eqref{eq9} show that we are dealing with rational, irrational and real numbers.
Therefore formulae like \eqref{eq9} hardly  highlight the geometrical aspects provided by their complex representation \eqref{eq6}.  Thus, the consideration of the independent real variable $t$, instead of $n$ used on the traditional Binet formulae \cite{hoggatt,devlin}, enables to extend the functions $g$, $h$ in \eqref{eq6}  to the geometrically rich set of complex numbers $\Cb$ (the geometric properties of standard complex\--valued maps is comprehensively treated in \cite{needham}).

 \medskip
 \noindent
 Let us begin by setting the notion of a three\--dimensional curve and define what we mean by a P\--Lucas or P\--Fibonacci curve.
  
 \begin{definition}\label{def2} (P\--Lucas/Fibonacci  curves)
 
 \noindent
 Given a function $f:\Rb\mapsto\Cb$, we define the curve $\gamma_f: \Rb\mapsto \Rb^3$ by $$\gamma(t)=\left(x(t),y(t),z(t)\right),$$
 where
  \begin{equation}\label{eq11}
  \left\{
 \begin{array}{l}
  x(t)= {\cal R}e\,  f(t) \\
 y(t)={\cal I}m\, f(t)\\
 z(t)=t,  \hspace{4cm} t\in \Rb\ .
 \end{array}
  \right.
 \end{equation}
When  $f=g$ or $f=h$, given in \eqref{eq6}, the corresponding curve \eqref{eq11} will be called a P\--Lucas or P\--Fibonacci curve. In the case the value of $P$ is fixed  we simply call the curve associated to $g$  (resp. $h$)  Lucas curve  (resp. Fibonacci curve).
  \end{definition}
  
 \medskip
 \noindent
 In the particular instance such that
\begin{equation}\label{eq11C}
\left\{
\begin{array}{l}
x(t)= {\cal R}e\,  f(t)= R\, \cos(\pi\, t)\\
y(t)= {\cal I}m\, f(t)= R\, \sin(\pi\, t) \\
z(t)=t\hspace{4cm} t\in \Rb,
 \end{array}
 \right.
\end{equation}
 where $R>0$, the function $f$ is representable by the 3\--dimensional helix of ratio $R$ and pitch $2$. Indeed, recall that
the parametric equations
 \begin{equation}\label{eq12}
 \left\{
 \begin{array}{l}
 x(\theta)=R\, \cos(\theta)\\
 y(\theta)=R\, \sin(\theta)\\
 z(\theta)=\displaystyle \frac{p}{2\, \pi} \, \theta,    \hspace{3cm} \theta\geq 0
 \end{array}
 \right.
 \end{equation}
 define the (rectangular) coordinates of the points of a helix with polar angle $\theta$, ratio $R>0$, and pitch $p>0$. 
 
 \medskip
  \noindent
Helices like \eqref{eq11C}  and other helicoidal or spiral curves are widely present in the nature as well as  in numerous mankind achievements (see for instance \cite{capanna}, \cite{watson}, \cite{heritage}, \cite{bridge}).

 \begin{remark}\label{remA}
  \end{remark}
  \noindent If the parameter $t$ in \eqref{eq11C} is an integer $k$ the component $z(k)=k$. Thus, in order that a point of the corresponding curve $\gamma_f$ be a lattice point (that is, a point with integer coordinates) we should have $x(k)$, $y(k)\in \Zb$. Therefore, a point ${\cal P}_j$ belonging to the helix \eqref{eq11C}  is a lattice point  if and only if $R\in \Zb$, in which case
\begin{equation}\label{eq11B}
{\cal P}_j=(\pm R,0,j),\qquad j\in \Zb\ .
\end{equation}

 \noindent
 Note that in \eqref{eq11C} we have considered the parameter $t\in \Rb$. In the examples given in Section \ref{sec3} we restrict $t\in[t_{min},t_{max}]$, for given integers $t_{min}<t_{max}$.   
 
 \medskip
 \noindent
 The value $p$ in \eqref{eq12} represents  the vertical displacement of a point in the helix after an angular increment of $2\, \pi$.
 The points of the helix lie in a cylinder with ratio $R$,  the point $(0,0,0)$ being the origin of the coordinate system. Of course, for $\theta=0$, the point $(R,0,0)$ is on the helix.
 \noindent
 Consequently, a parametric curve given by $x(t)=R\cos(\pi t)$, $y(t)=R\sin(\pi t)$, $z(t)=  t$, with $t_{min}\leq t\leq t_{max}$, represents the helix of polar ratio $R$ and pitch $2$, on  a cylinder of radius $R$ and height  $t_{max}-t_{min}$.

  \medskip
 \noindent In the following Proposition \ref{prop1} we discuss two maps $\psi_1$ and $\psi_2$ (depending on $g,h$ defined by \eqref{eq6}) which have the interesting property of being representable as helices.

 \begin{proposition}\label{prop1}
 Let $g$ and $h$ be respectively  P\--Lucas and P\--Fibonacci curves. Consider the complex\--valued functions
 \begin{equation}\label{eq13}
 \begin{array}{l}
 \psi_1(t)=g^2(t)- g(t-1)\, g(t+1)\\
 \psi_2(t)=h(t-1)\, h(t+1) - h^2(t),\, \qquad t\in \Rb\ .\\ 
 \end{array}
 \end{equation}

 \noindent
 (i) The curves of  $\psi_1$ and $\psi_2$
are helices of  polar angle $\theta=\pi\, t\,\,\mbox{(rad)}$ and (common) pitch $p=2$.
We call the curve associated to  $\psi_1(t) $ a P\--Lucas helix and the one associated to  $\psi_2(t) $  a P\--Fibonacci.

\medskip
\noindent
(ii)
For any  integer $P\geq1$ a P\--Lucas helix and a P\--Fibonacci one do not intersect.

\medskip
\noindent
(iii) \label{cor2}     For any  integer $P\geq1$, the P\--Lucas sequence verify the equalities
\begin{equation}\label{eq13a}
L_k^2-L_{k-1}\, L_{k+1}= (P^2+4)\, (-1)^k, \qquad k=0,1,2,\ldots
\end{equation}

\medskip
\noindent
(iv)  For any  integer $P\geq1$, the P\--Fibonacci sequence verify the equalities
\begin{equation}\label{eq13b}
F_{k-1}\, F_{k+1}- F_k^2=  (-1)^k, \qquad k=0,1,2,\ldots
\end{equation}
\end{proposition}
 
 \noindent
The  ratio $R$ of the helices corresponding to the maps $\psi_1$ and $\psi_2$ is summarised  in Table \ref{tab1} as well as the ratio of a helix corresponding to a map $\psi_3$ to be introduced in Section \ref{sec2}.

 \begin{proof}

\noindent
(i)   From \eqref{eq6} and \eqref{eq4}\--(iii), we have
\begin{equation}\label{eq14a}
\begin{array}{l}
g^2(t)=\lambda_1^{2 \, t}+ \lambda_2^{2\, t}+ 2\,( \lambda_1\, \lambda_2)^t= \lambda_1^{2 \, t}+ \lambda_2^{2\, t}+ 2\,(-1)^t,
\end{array}
\end{equation}
and
$$
g(t-1)\, g(t-2)=\lambda_1^{2\, t}+\lambda_2^{2\, t}+\lambda_1^{t-1}\, \lambda_2^{t-1}\, \left(\lambda_2^2+\lambda_1^2 \right)\ .
$$
Taking into account that $\lambda_1+\lambda_2=P$,  $\lambda_1^2=P\, \lambda_1+1$ and $\lambda_2^2=P\,\lambda_2+1$, we get
$$ \lambda_1^2 +\lambda_2^2 =P\, (\lambda_1+\lambda_2)+2=P^2+2\ .$$
Since $\lambda_1\, \lambda_2=-1$, one obtains
\begin{equation}\label{eq14b}
\begin{array}{ll}
g(t-1)\, g(t+1)&=\lambda_1^{2\, t}+\lambda_2^{2\, t}+ (P^2+2)\, (-1)^{t-1}\\
&=\lambda_1^{2\, t}+\lambda_2^{2\, t}- (P^2+2)\, (-1)^{t}\ .\\
\end{array}
\end{equation}
Thus, from \eqref{eq14a} and \eqref{eq14b}, we have
\begin{equation}\label{eq14}
\begin{array}{ll}
\psi_1(t)&= g^2(t)-  g(t-1)\, g(t+1) =(P^2+4)\, (-1)^t=D \, (-1)^t\\\
& =D\, \left(\cos(\pi\, t)+ i\,\sin(\pi\, t)\right),\qquad t\in \Rb\ .
\end{array}
\end{equation}
Therefore, the following  curve
\begin{equation}\label{eq15}
\left\{
\begin{array}{l}
x(t)= {\cal R}e\, (\psi_1(t))= D\, \cos (\pi\, t)\\
y(t)= {\cal I}m\, (\psi_1(t))= D\,\sin (\pi\, t)\\
z(t)=t,\hspace{6cm}\qquad t\in \Rb
\end{array}
\right.
\end{equation}
is a helix of ratio $D=P^2+4$ and pitch $2$.

\medskip
\noindent
From \eqref{eq6} we have
\begin{equation}\label{eq16}
D\, h^2(t)= \lambda_1^{2\, t}+\lambda_2^{2\, t}-2 \, (\lambda_1\, \lambda_2)^t= \lambda_1^{2\, t}+\lambda_2^{2\, t}-2 \, (-1)^t
\end{equation}
and 
\begin{equation}\label{eq17}
\begin{array}{ll}
h(t-1)\, h(t+1)&=\displaystyle \frac{1}{D}\left( \lambda_1^{2\, t}+ \lambda_2^{2\, t} -(\lambda_1\, \lambda_2)^{t-1}\, (\lambda_1^1+\lambda_2^2)\right)\\
\\
&=\displaystyle \frac{1}{D}\left( \lambda_1^{2\, t}+ \lambda_2^{2\, t} - (P^2+2)\,(\lambda_1\, \lambda_2)^{t-1} \right)\\
\\
&=\displaystyle \frac{1}{D}\left( \lambda_1^{2\, t}+ \lambda_2^{2\, t} - (P^2+2)\,(-1)^{t-1} \right)\ .\\
\end{array}
\end{equation}
Thus,
$$
D\, \left( h(t-1)\, h(t+1)   \right)=  \lambda_1^{2\, t}+ \lambda_2^{2\, t}+ (P^2+2)\, (-1)^t,
$$
and so
\begin{equation}\label{eq18}
D\, \left( h(t-1)\, h(t+1) -h^2(t)  \right)= (P^2+4)\, (-1)^t=D\, (-1)^t\ .
\end{equation}
That is,
\begin{equation}\label{eq19}
\psi_2(t)= h(t-1)\, h(t+1) -h^2(t) =  (-1)^t.
\end{equation}
Using the same reasoning as before, the last equality shows that $\psi_2(t)$ can be written as  a hélix with ratio $R=1$ and pitch $2$.

\medskip
\noindent
(ii)
From \eqref{eq14} and \eqref{eq19} we conclude that
\begin{equation}\label{eq20}
\psi_1=\, D\, \psi_2(t)= (P^2+4)\, \psi_2(t), \qquad t\in \Rb \ .
\end{equation}
This means that a Lucas helix has a ratio that is $P^2+4$ greater than the ratio of the Fibonacci helix and the conclusion follows.

\medskip
\noindent
(iii) Follows from\eqref{eq14} for  $t=0,1,2,\ldots\ .$

\medskip
\noindent
(iv)
Follows from \eqref{eq19} for $t=0,1,2,\ldots\ .$
\end{proof}

\begin{table}
$$
\begin{array}{|l|c|c| }
\hline
\hspace{2cm} \mbox{Map}& R\,\, \mbox{(Ratio)} & \mbox{Helix name}\\
 \hline
\psi_1(t)=g^2(t)- g(t-1)\, g(t+1)&P^2+4=D&\mbox{Lucas}\\
 \hline
\psi_2(t)=h(t-1)\, h(t+1) -h^2(t)& 1&\mbox{Fibonacci}\\
 \hline
 \psi_3(t)=g^2(t)- D\, h^2(t)& 4&\mbox{Pell}\\
  \hline
\end{array}
$$
 \caption{Lucas/Fibonacci/Pell helices.  
  \label{tab1} }
\end{table}

 \begin{remark}\label{rem4} 
  \end{remark} 
  
  \noindent
  The curves $\psi_1$ and $\psi_2$ defined in \eqref{eq13} are not the unique helices one can associate respectively to Lucas or Fibonacci sequences. For instance, from
 $$
 g(2\, t)=\lambda_1^{2\, t}+\lambda_2^{2\, t}
 $$
 and
 $$
 g^2(t)=\lambda_1^{2\, t}+\lambda_2^{2\, t}+ 2\, (\lambda_1\, \lambda_2)^t
 $$
 we get
 $$
 g(2\, t)-g^2(t)= -2\, (\lambda_1\, \lambda_2)^t= 2\, (-1)^t
 $$
 and so
 \begin{equation}\label{eq41}
  {\cal L} (t)= g(2\, t)-g^2(t),
   \end{equation}
   is a map representable as a helix (with ratio $R=2$) also associated to P\--Lucas sequences.
   
   \medskip
   \noindent
Also, from \eqref{eq41} we get
\begin{equation}\label{eq42}
L_k^2-L_{2\,k}= 2\, (-1)^k, \qquad k\in \Zb
\end{equation}
which is an extension of a well known relationship among terms of  P\--Lucas sequences (\cite{cohen}, p. 421, Pro\-po\-si\-tion 6.8.1). How\-ever, in this pre\-li\-mi\-na\-ry work, we will not discuss further matters related to the map   ${\cal L}$ given in \eqref{eq41} as well as  other elaborations on the referred unicity issue.

\medskip
\noindent
  New helices associated to the maps $\Psi_1$ and $\Psi_2$ will be defined in the next section.

\section{Some helices related to diophantine equations of Pell }\label{sec2}

The next definition characterizes a certain helix which we name  Pell's helix  due to its connection to two famous quadratic diophantine equations, namely  $u^2- b\, v^2=4$ and $u^2- b\, v^2=-4$, where $b>0$ is a no perfect square integer  and $u,v$ are integers (for the Pell\--Fermat equation $u^2- b\, v^2=N$ see for instance \cite{cohen}, p. 354).

\medskip
\noindent
Consider the complex\--valued function
\begin{equation}\label{eq21}
\psi_3(t)=g^2(t)- D\, h^2(t), \qquad t\in \Rb\ .
\end{equation}
\noindent
 As shown in Proposition \ref{prop3} the function $\psi_3$ can be written as a helix of ratio $R=4$.

 \begin{definition}\label{def4}
A Pell's helix (of ratio 4), is a helix  parametrized by
 \begin{equation}\label{eq22}
 \begin{array}{l}
 x(t)= 4\, \cos(\pi\, t)\\
 y(t)=4\, \sin (\pi\, t)\\
 z(t)=  t, \qquad\qquad  t\in \Rb\ .
 \end{array}
 \end{equation}.
 \end{definition} 
\begin{proposition}\label{prop3}
Let $L_k$ and $F_k$ be the $k$\--th terms of the  sequences given in \eqref{eq9} and $\psi_3$ as in \eqref{eq21}.

\noindent
(i) The function $\psi_3$ can be written as  Pell's helix  \eqref{eq22}.

\medskip
\noindent
(ii) For $j$ integer, the following points ${\cal P}$ and ${\cal Q}$  are the only lattice points for the Pell's helix \eqref{eq22}:
\begin{equation}\label{eq23}
\begin{array}{ll}
{\cal P}_j=(4,0, j)&\mbox{for}\,\, j\,\, \mbox{even} \\
{\cal Q}_j=(-4,0, j)&\mbox{for}\,\, j\,\, \mbox{odd} \ .\\
\end{array}
\end{equation}

\medskip
\noindent
(iii) The point $(L_k,F_k)$ satisfies respectively the diophantine equation $x^2-D\, y^2=4$ for $k$ even, and  $x^2-D\, y^2=-4$ for $k$ odd.
\end{proposition}
\begin{proof} 

\noindent
(i)  From \eqref{eq14a} and \eqref{eq16} we have
\begin{equation}\label{eq24}
g^2(t)- D\, h^2(t) = 4\, (-1)^t,
\end{equation}
and so the function $\psi_3$ can be written as the helix given by \eqref{eq22}.

\medskip
\noindent
(ii) We have here a particular case of the curve in \eqref{eq11C} with $R=4$ or $R=-4$ (see remark \ref{remA}). So the lattice points corresponding to \eqref{eq11B} are giving  by \eqref{eq23}.

\medskip 
\noindent
(iii) For $t\in \Zb$ the numbers $g(t)$, $D$ and $h(t)$ are integers.
When $t$ is even, say $t=2\,k$ with $k\in \Zb$, let $x=g(t)$, $y=h(t)$. By \eqref{eq24} we get $x^2-D\, y^2=4$. When $t$ is odd the same reasoning leads to  $x^2-D\, y^2=-4$.
\end{proof}

\noindent
Considering the maps given in Table \ref{tab1} one concludes immediately that a suitable linear combination of $\psi_1$ to $\psi_3$ leads to new maps \--- denoted by $\psi_4$ to $\psi_7$ in Table \ref{tab2} \--- which are also representable by helices,  whose ratio $R_m$ is given in the second column of that table.
Translating the tabulated maps by its expressions in terms of the maps $g$ and $h$, by direct computation from the definitions given for the maps $\psi_m$, for $m=4$ to $m=7$,  immediately one justifies the following proposition and so we omit its proof. Of course other helices, related to P\--Lucas/Fibonacci/Pell sequences, having ratio  congruent to $0,1,2,\ldots, P^2+3$ $ \pmod D$, could also be defined by means of a convenient linear combination of $\psi_1, \psi_2$ and $\psi_3$.
\begin{table}
$$
\begin{array}{ |l |c| c |}
\hline
\hspace{2cm} \mbox{Map}& R_m\,\, \mbox{(Ratio)} & \mbox{Connection}\\
 \hline
\psi_4(t)=\psi_1(t)-(D-1)\, \psi_2(t)&1& \mbox{Fibonacci\--Lucas}\\
 \hline
\psi_5(t)= \psi_3(t)- 2\, \psi_2(t)& 2& \mbox{Pell\--Fibonacci}\\
 \hline
 \psi_6(t)=\psi_3(t)- \psi_2(t)& 3& \mbox{Pell\--Fibonacci}\\
  \hline
   \psi_7(t)=2 D\, \psi_2(t)- \psi_1(t)& D=P^2+4& \mbox{Fibonacci--Lucas}\\
  \hline
\end{array}
$$
 \caption{Maps $\psi_4$ to $\psi_7$.  
  \label{tab2} }
\end{table}

\begin{proposition}\label{prop4} The following complex\--valued maps $\psi_4$ to $\psi_7$, expressed in terms of $g$, $h$ given in \eqref{eq6}, are all representable by helices whose ratio $R_m$ is summarized  in Table \ref{tab2}.
\begin{equation}\label{eq30}
\begin{array}{l}
\psi_4(t)=\psi_1(t)- (D-1)\, \psi_2(t)=\\
\hspace{1cm} =g^2(t)-g(t-1)\, g(t+1)- (P^2+3)\, \left( h(t-1)\, h(t+1)-h^2(t)    \right) =\\
\hspace{1cm}=(-1)^t,   \\
\\
\psi_5(t)=\psi_3(t)-2\, \psi_2(t)=\\
\hspace{1cm} =g^2(t)-(P^2+2)\, h^2(t)- 2\, h(t-1)\, h(t+1)= 2\, (-1)^t,    \\
\\
\psi_6(t)= \psi_3(t)- 2\, \psi_2(t) =   \\
\hspace{1cm} =g^2(t)-(P^2+4)\, h^2(t)-h(t-1)\, h(t+1)= 3\, (-1)^t,   \\
\\
\psi_7(t)= 2D\, \psi_2(t)-\psi_1(t))=   \\
\hspace{1cm} =g(t-1)\,g(t+1)- (g^2(t)+h^2(t))+2\, h(t-1)\, h(t+1)=\\
\hspace{1cm} =D\, (-1)^t=(P^2+4)\, (-1)^t, \qquad t\in \Rb\ .
\end{array}
\end{equation}

\end{proposition}
\begin{cor}\label{cornew} Replacing the real variable $t$ in \eqref{eq30}, by the integer variable $k\in \Zb$, the P\--Lucas/Fibonacci sequences satisfy the following equalities
\begin{equation}\label{eq31}
\psi_m(k)=R_m\, (-1)^k,\qquad m=4,\ldots,7,
\end{equation}
where $R_m$ denotes the ratio of the helix corresponding to the map $\psi_m$ displayed in Table \ref{tab2}. That is,
\begin{equation}\label{eq32}
\begin{array}{l}
L_k^2- L_{k-1}\, L_{k+1}- (P^2+3)\, (F_{k-1}\, F_{k+1}-F_k^2) =(-1)^k,\\
\\
L_k^2- (P^2+2)\, F_k^2-2\, F_{k-1}\, F_{k+1}=2\, (-1)^k,\\
\\
L_k^2- (P^2+3)\, F_k^2-2\, F_{k-1}\, F_{k+1}=3\, (-1)^k,\\
\\
L_{k-1}\, L_{k+1}-(L_k^2+F_k^2)+ 2\, F_{k-1}\,F_{k+1}=(P^2+4) \, (-1)^k\qquad k\in \Zb\ . \\
\end{array}
\end{equation}
\end{cor}

\noindent
The literature on Lucas and Fibonacci numbers is vast. Restricting $k$ in \eqref{eq32}  to non\--negative integers one obtains relations between Lucas and Fibonacci numbers which are certainly known in the case of $P=1$.


 \begin{figure}[h]
\begin{center} 
 \includegraphics[totalheight=7.0cm]{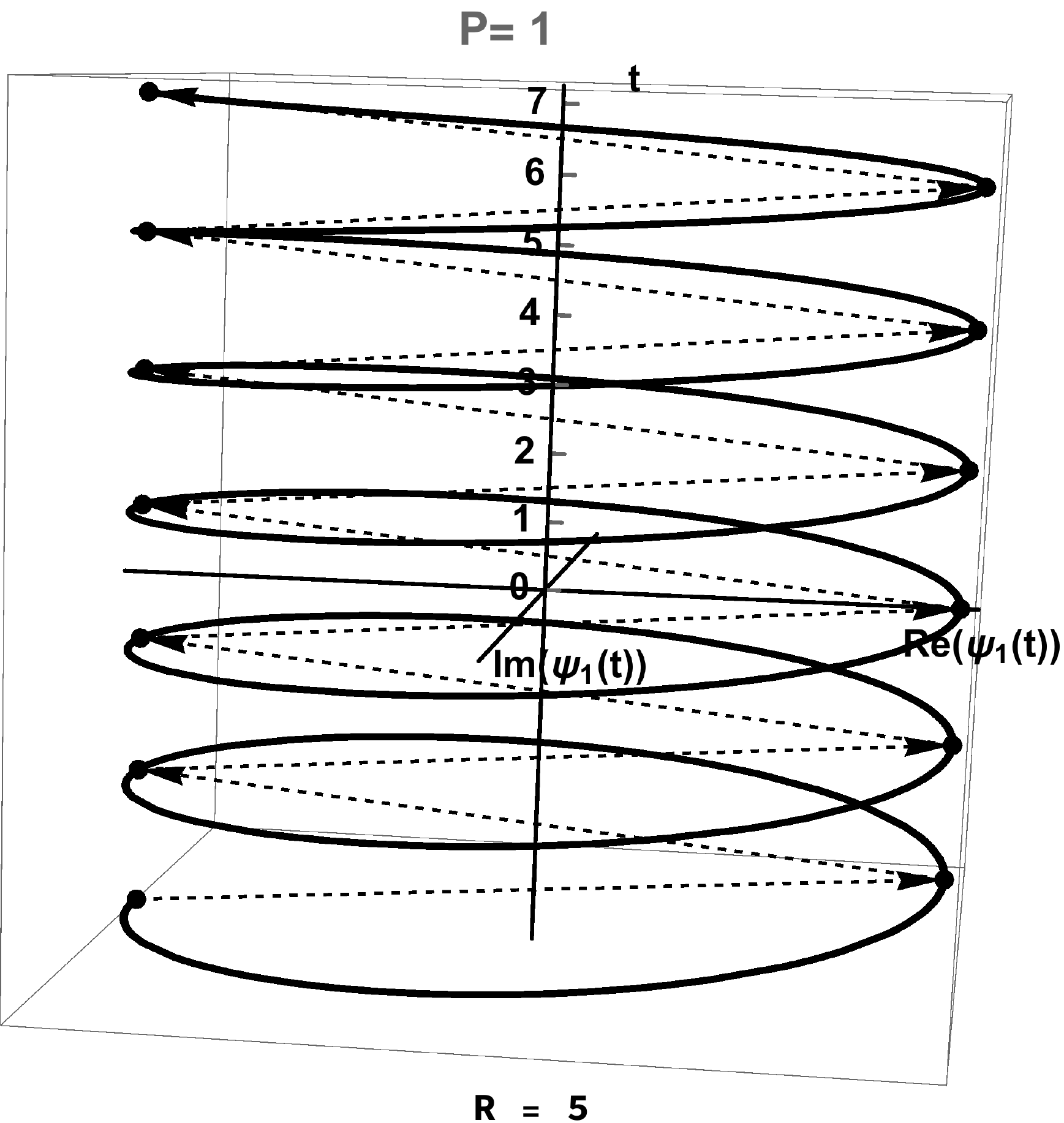}
\caption{\label{fig1}  The helix for the function  $\psi_1(t)$, with $P=1$ and  $-5\leq t\leq 7$.}
\end{center}
\end{figure}

  \begin{figure}[h]
\begin{center} 
 \includegraphics[totalheight=7.0cm]{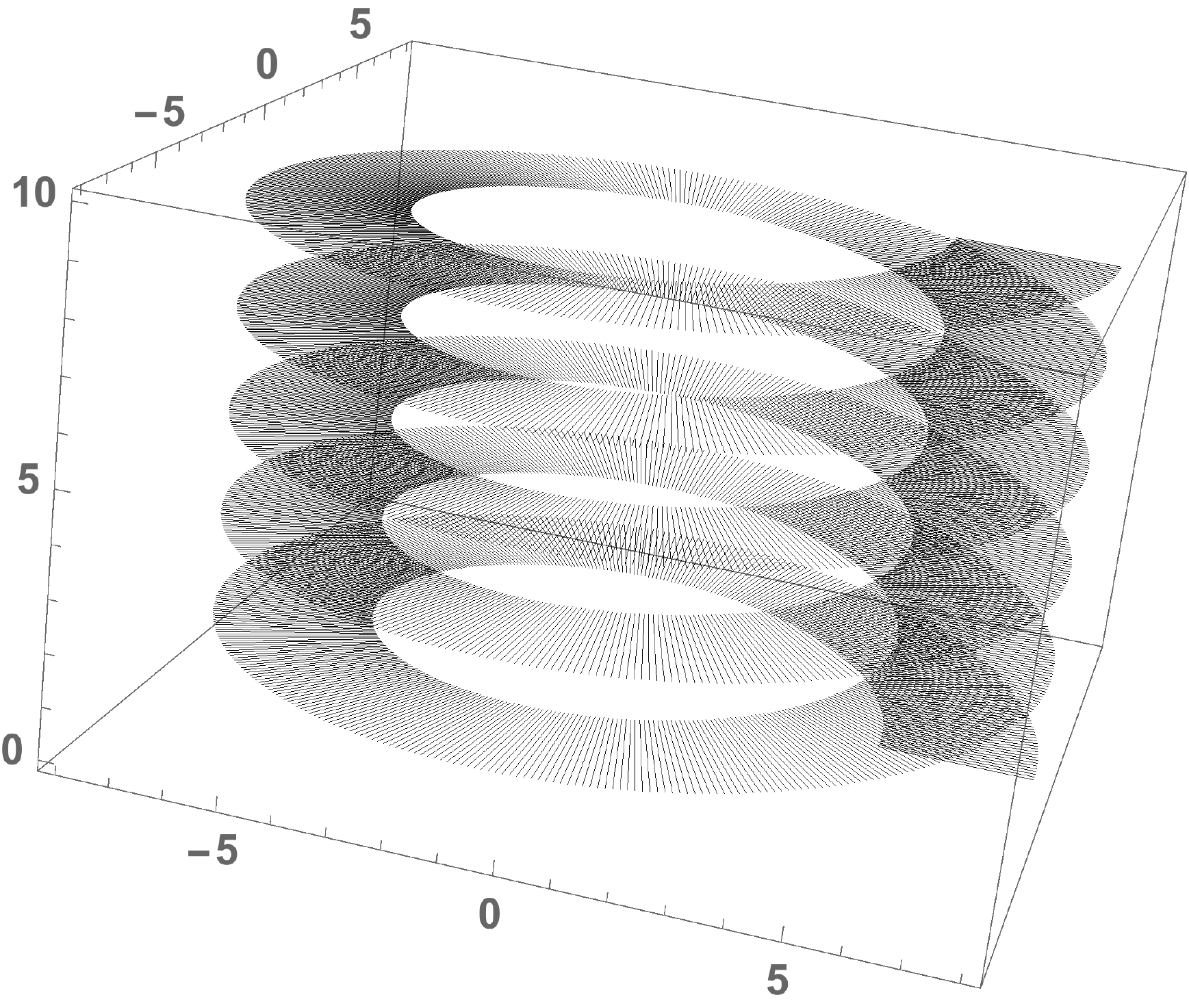}
\caption{\label{fig2}  Double helix: 1\--Lucas and 2\--Lucas  $(0\leq t\leq 10$).}
\end{center}
\end{figure}

\section{Examples}\label{sec3}
In the next examples we apply the complex expressions  $g(t),h(t), \psi_1(t),\psi_2(t)$ and $\psi_3(t)$, respectively defined in \eqref{eq6}, \eqref{eq13} and \eqref{eq21},  for several values of the integer parameter $P\geq 1$.

\medskip
\noindent
 Given integers $t_{min}< t_{max}$ we considerer the interval $[t_{min}, t_{max}]$ for domain of the real independent variable $t$. The corresponding curves has been obtained by computing the real and imaginary parts of the maps $g, h, \psi_1,\psi_2$ and $\psi_3$. The computer algebra system {\sl Mathematica} \cite{wolfram1} has been used to display the corresponding helical curves, by means of the built\--in commands $\verb+Re[ ]+$, $\verb+Im[ ]+$ and $\verb+ParametricPlot3D[ ]+$.
 
 \medskip
 \noindent
 The resulting  helices are presented in order to confirm and illustrate  Propositions \ref{prop1} and \ref{prop3}. We also introduce some double helices resulting from a combination of Lucas/Fibonacci helices. Other double and triple helices will be discussed in a forthcoming work.

\begin{exemplo}\label{exemplo1}
\end{exemplo}

\noindent
For $P=1$ we have $D=P^2+4=5$.  Lucas sequence is $$2,1,3,4,7,11,18,29,\ldots,$$ and  the Fibonacci sequence is $$0,1,1,2,3,5,8,13,21,\ldots \ .$$

\noindent
The respective maps $g$ and $h$ are
\begin{equation}\label{eq25}
\begin{array}{l}
g(t)= \displaystyle  \frac{1}{2^t}\left( (1+\sqrt{5})^t+  (1-\sqrt{5})^t \right)\\
\\
h(t)= \displaystyle  \frac{1}{\sqrt{5}\, 2^t}\left( (1+\sqrt{5})^t-  (1-\sqrt{5})^t \right), \qquad t\in \Rb \ .
\end{array}
\end{equation}

\noindent
The functions $\psi_1$, $\psi_2$ and $\psi_3$ have the following expressions:
\begin{equation}\label{eq26}
\begin{array}{ll}
\psi_1(t)=g^2(t)-g(t-1)\, g(t-2)= 5\, (-1)^t&\mbox{ (Lucas)}\\
\psi_2(t)=h(t-1)\, h(t+1)- h^2(t)= (-1)^t &\mbox{ (Fibonacci)} \\
\psi_3(t)=g^2(t)- D\, h^2(t)= 4\, (-1)^t, \quad t\in \Rb\quad &\mbox{ (Pell),}\\
\end{array}
\end{equation}
which show that these maps are representable as helices with ratios 5, 1 and 4 respectively. 

\medskip 
 \noindent
 For $P=1$, the Lucas helix is displayed in Figure \ref{fig1}, where the respective  lattice points are shown by means of bold black points. As theoretically expected it is clear that the displayed dashed arrows connecting two consecutive lattice points belong to the vertical plane defined by the $x$ and $z$ axes.

 \begin{exemplo}\label{exemplo2}
  \end{exemplo}
 \begin{figure}[h]
\begin{center} 
 \includegraphics[totalheight=7.4cm]{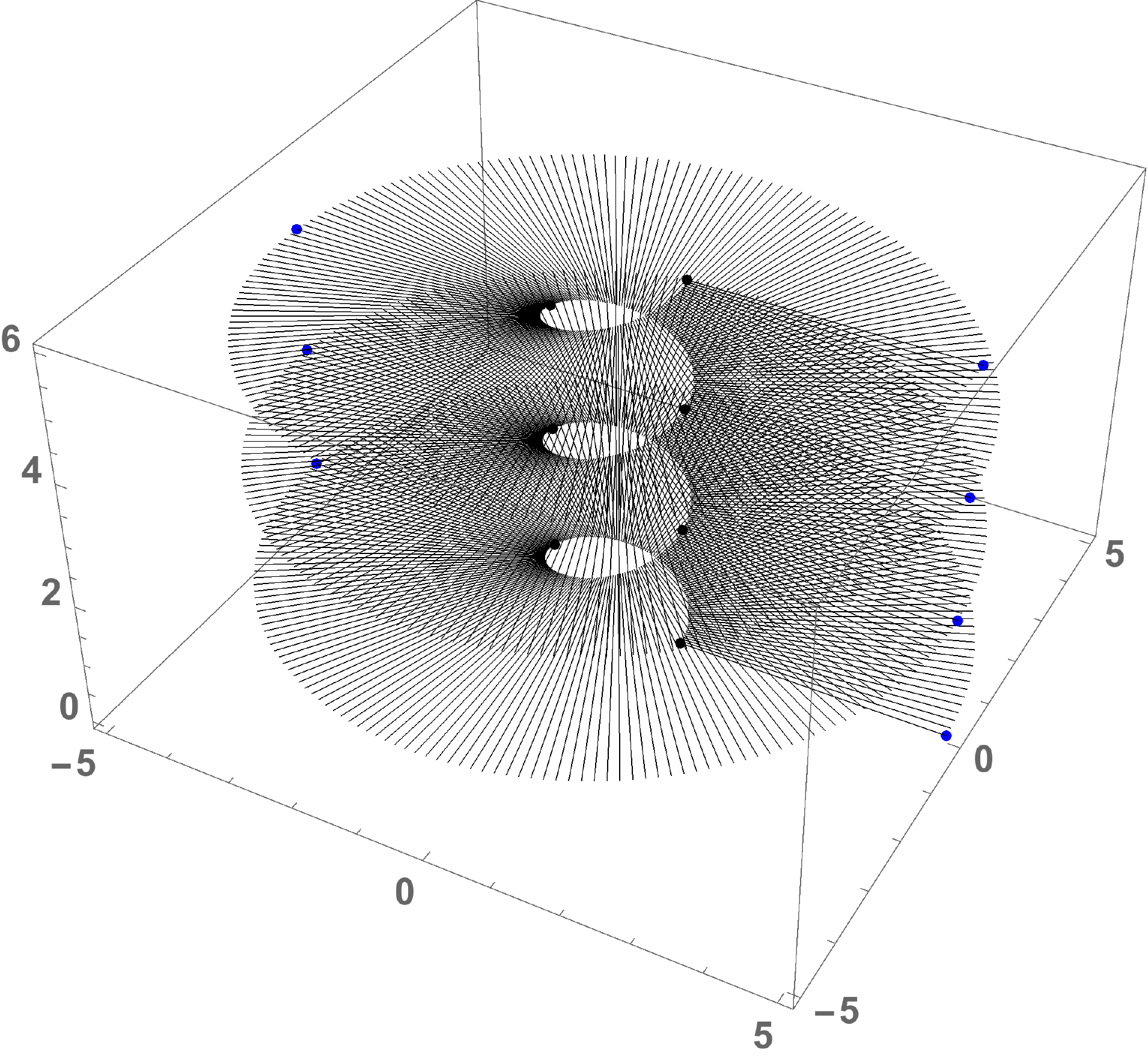}
\caption{\label{fig3}  Double helix: 1\--Lucas to 2\--Fibonacci.}
\end{center}
\end{figure}

\noindent
 In Figure \ref{fig2} we show a double helix. It results from connecting the point $\psi_1(t)$ with $P=1$ to the point $\psi_1(t)$  with  $P=2$, for $t\in[0,10]$. The respective graph was obtained discretising the variable $t$ taking  $t_{min}=0$ to $t_{max}=10$ and increments $dt=0.005$.  The inner helix (with $P=1$) has ratio $R=p^2+4=5$, while the outer helix (with $P=2$) has ratio $2^2+4=8$, as expected.

 \newpage
 \begin{exemplo}\label{exemplo3}
 \end{exemplo}

 \noindent
Using analogous technique as in Example \ref{exemplo2},  the next double helix (Figure~\ref{fig3}) connects 
a 1\--Lucas helix to a 2\--Fibonacci helix. Thus the resulting double helix has interior ratio $R=1$ (for 2\--Fibonacci) and exterior ratio $R=5$ (for 1\--Lucas). Figure \ref{fig3A} displays a lateral (right) view of the helix in Figure \ref{fig3}, confirming that its lattice points  belong to the vertical plane $xz$.
 \begin{figure}[h]
\begin{center} 
 \includegraphics[totalheight=5.2cm]{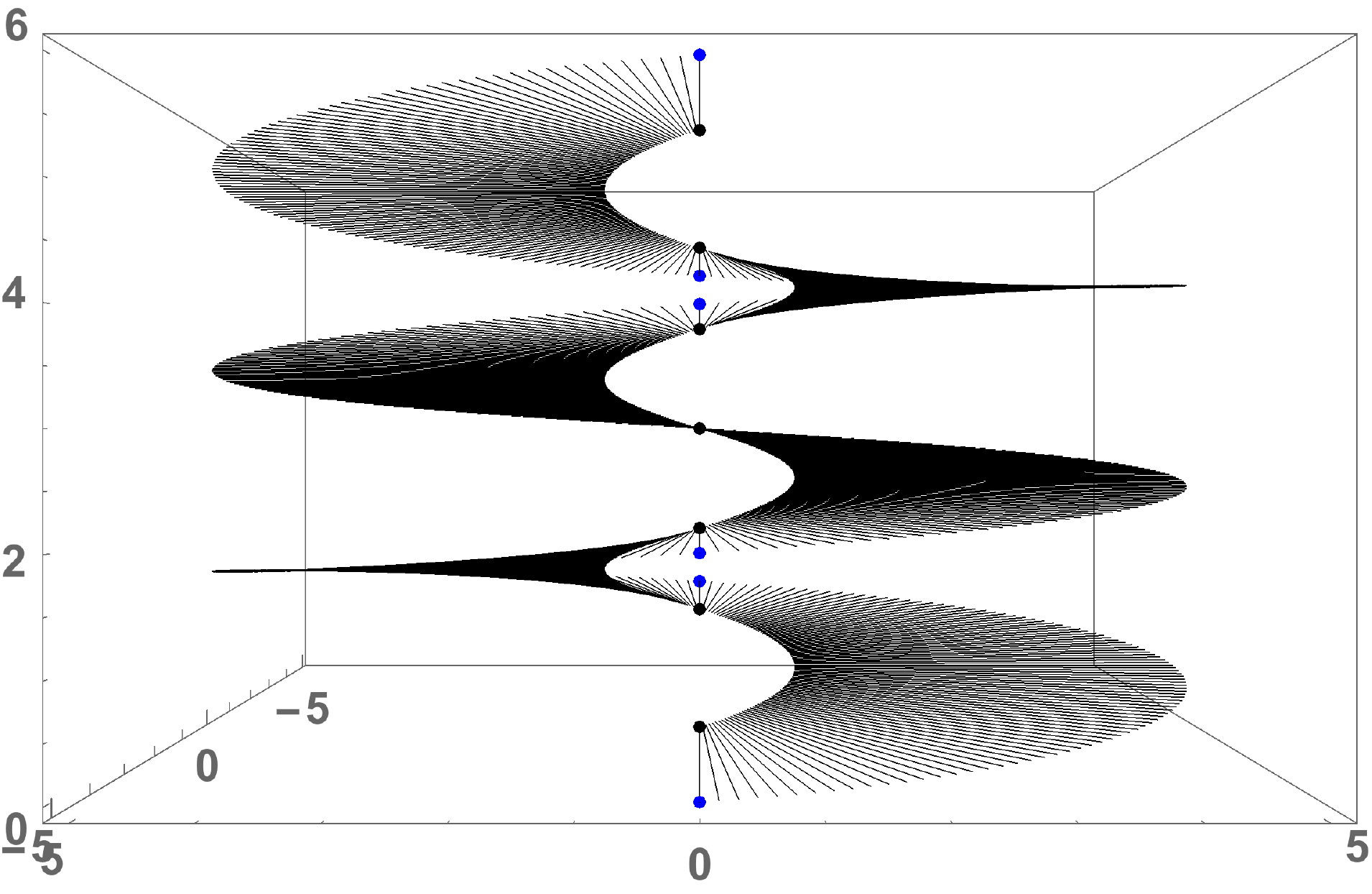}
\caption{\label{fig3A}  Double helix: 1\--Lucas to 2\--Fibonacci  (right view).}
\end{center}
\end{figure}
 \begin{figure}[h]
\begin{center} 
 \includegraphics[totalheight=7.0cm]{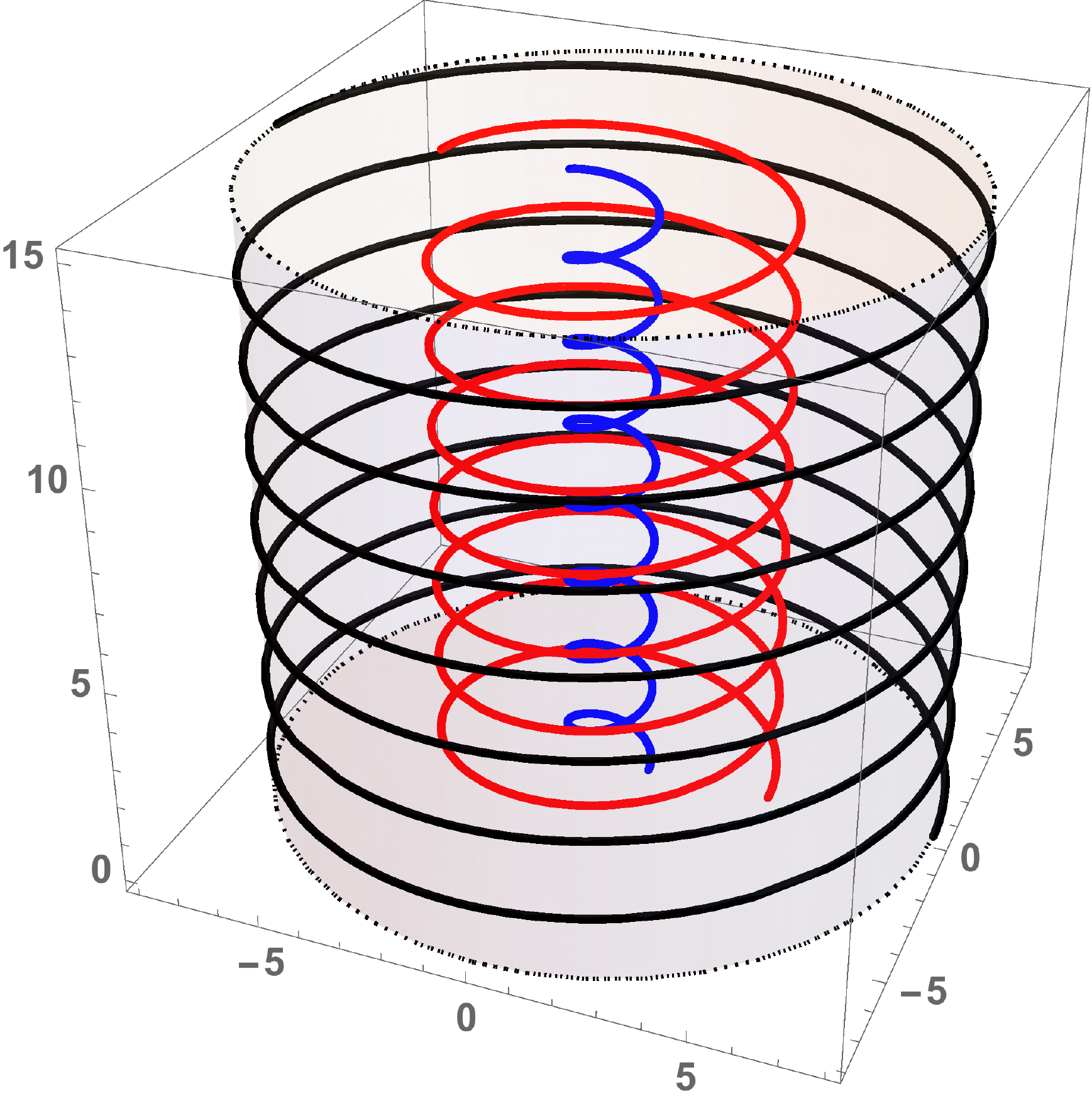}
\caption{\label{fig4}  Helices of Lucas (black), Pell (red) and Fibonacci (blue), for $P=2$.}
\end{center}
\end{figure}

\begin{exemplo}\label{exemplo4}
\end{exemplo}

\noindent In Figure \ref{fig4} we show a cylinder containing three helices, for $0\leq t\leq 15$ and  $P=2$. The outer helix results from $\psi_1(t)= 8\, (-1)^t$ (Lucas) and the inner one from $\psi_2(t)= (-1)^t$ (Fibonacci). As expect the middle helix corresponds to  $\psi_3(t)= 4\, (-1)^t$ (Pell). 
The ratios of these helices are  in accordance to the values of $R$ in  Table \ref{tab1}. 

\end{document}